\documentclass[11pt]{amsart}

\usepackage{amsmath,amsthm,amsfonts,amssymb,mathrsfs}
\usepackage[margin=1in]{geometry}
\usepackage[bookmarks]{hyperref}
\usepackage{verbatim}
\usepackage{todonotes}

\usepackage{fancyvrb }
\usepackage{color}

\newenvironment{enumalph}{\begin{enumerate}  }{\end{enumerate}}

\newtheorem{theorem}{Theorem}[subsection]

\newtheorem{proposition}[theorem]{Proposition}

\newtheorem{corollary}[theorem]{Corollary}

\newtheorem{lemma}[theorem]{Lemma}
\input xy
\xyoption{all}

\theoremstyle{definition}

\newtheorem{remark}[theorem]{Remark}

\DeclareMathOperator{\rank}{rank}

\DeclareMathOperator{\chr}{char}

\DeclareMathOperator{\opH}{H}

\DeclareMathOperator{\Lie}{Lie}

\newcommand{\calO}{\mathcal{O}}

\newcommand{\A}{\mathbb{A}}

\newcommand{\spec}{\mbox{{\rm{Spec}}\;}}
\newcommand{\red}{\rm{red}}
\newcommand{\reg}{\rm{reg}}
\newcommand{\subreg}{\rm{subreg}}
\newcommand{\sub}{\rm{sub}}

\newcommand{\frakb}{\mathfrak{b}}
\newcommand{\g}{\mathfrak{g}}

\newcommand{\fraku}{\mathfrak{u}}

\newcommand{\fraksl}{\mathfrak{sl}}

\newcommand{\N}{\mathcal{N}}

\newcommand{\Fp}{\mathbb{F}_p}

\begin{document}

\title[Mixed Commuting varieties over simple Lie algebras]{Mixed Commuting varieties over simple Lie algebras}

\author{Nham V. Ngo}
\address{Department of Mathematics\\ University of Arizona\\ Tucson\\ AZ 85721, USA}
\email{nhamngo@math.arizona.edu}

\maketitle

\begin{abstract}
Let $\g$ be a simple Lie algebra defined over an algebraically closed field $k$ of characteristic $p$. Fix an integer $r>1$ and suppose that $V_1,\ldots,V_r$ are irreducible closed subvarieties of $\g$. Let $C(V_1,\ldots,V_r)$ be the closed variety of all the pairwise commuting elements in $V_1\times\cdots\times V_r$. This paper studies the dimension and irreducibility of such varieties with various $V_i$ in a Lie algebra $\g$. In particular, we complete the problem for the case when $V_i$'s are either $\overline{\calO_{\sub}}$ the closure of the subregular orbit or $\N$ the nilpotent cone of any rank two Lie algebra $\g$. A result on the dimension of these mixed commuting varieties is generalized for higher rank $\g$. Finally, we apply our calculations to study properties of support varieties for a simple module over the $r$-th Frobenius kernels of $G$.      
\end{abstract}

\section{Introduction}
\subsection{} Let $k$ be an algebraically closed field of characteristic $p\ge 0$. Suppose $\g$ is a simple Lie algebra defined over $k$. For each $r\ge 2$, let $V_1,\ldots, V_r$ be irreducible closed subvarieties of a Lie algebra $\g$. Define
\[ C(V_1,\ldots, V_r)=\{(v_1,\ldots, v_r)\in V_1\times\cdots\times V_r~|~[v_i,v_j]=0~,~1\le i\le j\le r \}, \]   
a {\it mixed commuting variety} over $V_1,\ldots,V_r$\footnote{We are aware of the same terminology arising in \cite[Section 4]{Ya:2009}. The definition of mixed commuting varieties in that paper is entirely different from ours.}. Note that if $V_1=\cdots=V_r$, then this variety becomes $C_r(V_1)$, the commuting variety of $r$-tuples over $V_1$. Mixed commuting varieties of two tuples were first introduced in \cite[9.4]{Vas:1994}, the $r$-tuple version was defined and studied in \cite{N:2012}. In particular, the author explicitly described the irreducible decomposition for any mixed commuting variety over $\fraksl_2$ and its nullcone $\N$. The result implies that such varieties are mostly not Cohen-Macaulay or normal. 

In general, mixed commuting varieties are still mysterious. Interesting questions include
\begin{enumerate}
\item What is the dimension of $C_r(V_1,\ldots,V_r)$?
\item What are the irreducible components of $C_r(V_1,\ldots,V_r)$? 
%\item What is the singular locus of $C_r(V_1,\ldots,V_r)$?
\end{enumerate}

The results in this paper were motivated by investigating the cohomology for Frobenius kernels of algebraic groups. The first connection between commuting varieties and support varieties in cohomology of Frobenius kernels was constructed by Suslin, Friedlander, and Bendel in their two papers \cite{BFS1:1997}, \cite{BFS2:1997}. To be precise, let $G$ be a simply-connected simple algebraic group defined over $k$, and let $G_{(r)}$ be the $r$-th Frobenius kernel of $G$. Then there is a homeomorphism between the maximal ideal spectrum of the cohomology ring for $G_{(r)}$ and the nilpotent commuting variety over the Lie algebra $\g=$Lie$(G)$ whenever the characteristic $p$ is large enough. This variety is also the ambient space for support varieties of the $r$-th Frobenius kernels of $G$. In the case $r=1$, these support varieties for modules $L(\lambda)$ and $\opH^0(\lambda)$ are explicitly described by the work of Drupieski, Nakano, Parshall, and Vella \cite{NPV:2002},\cite{DNP:2012}. Sobaje uses these descriptions to compute the support varieties of $L(\lambda)$ for higher $r$ \cite{So:2012}. Our computation on mixed commuting varieties in the present paper gives the complexity of all simple module $L(\lambda)$ for groups $G$ of rank 2 and of some special cases for groups of higher rank.  Moreover, in the case when $\g=\fraksl_3$, we are able to compute the dimension of mixed commuting varieties involving $V_i=\g$, which deduces the dimension of the commuting variety $C_r(\mathfrak{g})$.

\subsection{Main results}\label{main results} 
The paper is organized as follows. We first review terminology and notation in Section \ref{notation}. Then in Section \ref{3 by 3 matrices}, we study the properties of $C_r(z_{\sub})$ where $z_{\sub}$ is the centralizer of the subregular nilpotent element corresponding to the partition $[n-1, 1]$ in $\mathfrak{gl}_n$. In particular, let $\g=\fraksl_n$ and $p\nmid n$. Then we prove that for every $r\ge 1$, $C_r(z_{\sub})$ is a product of an affine space of dimension $(n-2)r$ with a determinantal variety; hence it is normal and Cohen-Macaulay (cf. Theorem \ref{commuting variety over z_sub}). Next, we apply our calculations to compute the dimensions of $C_r(\fraksl_3)$ and $C_r(\mathfrak{gl}_3)$. Note that the latter variety was proved to be irreducible by Kirillov and Neretin \cite{KN:1987}, see also \cite{Gu:1992}; hence the dimension easily follows. Our method does not depend on the irreducibility of this variety. We also emphasize the impact of the characteristic $p$, i.e. $p\mid n$, on the commuting variety $C_r(\fraksl_3)$ (cf. Remark \ref{remark2 p=3}).

In Section \ref{mixed commuting varieties}, we investigate mixed commuting varieties over $\fraksl_3$ with $V_i$'s either $\overline{\calO}_{\sub}$, $\N$, or $\fraksl_3$. Our key ingredient is determinantal varieties over certain matrices. In particular, by analyzing the intersections $z_{\sub}\cap\overline{\calO_{\sub}}$ and $z_{\sub}\cap\N$, we reduce to the problem of computing the dimension of the varieties generated by 2 by 2 minors of the following matrix of indeterminates

\[ \left( \begin{array}{ccccccccc} x_1 & \cdots & x_i & x_{i+1} & \cdots & x_{i+j} & x_{i+j+1} & \cdots & x_{i+j+m} \\
0 & \cdots & 0 & y_1 & \cdots & y_j & y_{j+1} & \cdots & y_{j+m} \\ 
0 & \cdots & 0 & 0 & \cdots & 0 & z_1 & \cdots & z_m \end{array} \right). \]
We can also generalize this result for any arbitrarily large matrix (cf. Theorem \ref{general matrix}), which gives an interesting property of determinantal rings. Moreover, we are able to determine whether a mixed commuting variety is irreducible or not (cf. Theorems \ref{main theorem} and \ref{irreduciblity}).

We extend our calculations for arbitrary simple Lie algebras of rank 2 in the next section (\S \ref{other rank 2 groups}). In particular, for $g$ either of type $C_2$ or $G_2$, we compute the dimension of mixed commuting varieties of the form $C(\overline{\calO_{\sub}},\ldots,\overline{\calO_{\sub}},\N,\ldots,\N)$ and determine whether they are irreducible or not. The strategy is the same as in previous section, however, the computation is rather different. While the result for type $C_2$ is almost straightforward from calculations in \cite{N:2014}, that for the exceptional type $G_2$ contains complication as we need to employ a technique in Premet paper \cite{Pr:2003} (cf. Theorems \ref{Type C_2} and \ref{Type G_2}). One interesting fact from our results in this section is that all of the aforementioned mixed commuting varieties have the same dimension (except for a couple of cases). This is then proved in Section \ref{general result}, see Proposition \ref{generalization}.

In the last section, we use results in the previous sections to compute the dimension of support varieties of Frobenius kernels for a simple module $L(\lambda)$ in a certain case. From results in \cite{DNP:2012} and \cite{So:2012}, we point out a connection between support varieties and mixed commuting varieties (cf. Proposition \ref{connection}). Then using an explicit calculation of support varieties of the first Frobenius kernels in \cite{NPV:2002} we are able to describe all support varieties of the $r$-th Frobenius kernels of $G$ of rank 2. For higher rank groups, we classify simple modules $L(\lambda)$ such that their complexity is the same as that of $k$.
  
\section{Notation}\label{notation}

\subsection{Root systems and combinatorics}\label{combinatorics} Let $k$ be an algebraically closed field of characteristic $p$. Let $G$ be a simple, simply-connected algebraic group over $k$, defined and split over the prime field $\Fp$. Fix a maximal torus $T \subset G$, also split over $\Fp$, and let $\Phi$ be the root system of $T$ in $G$. Fix a set $\Pi = \{ \alpha_1,\ldots,\alpha_n \}$ of simple roots in $\Phi$, and let $\Phi^+$ be the corresponding set of positive roots. Let $B \subseteq G$ be the Borel subgroup of $G$ containing $T$ and corresponding to the set of negative roots $\Phi^-$, and let $U \subseteq B$ be the unipotent radical of $B$. Set $\g = \Lie(G)$, the Lie algebra of $G$, $\frakb = \Lie(B)$, $\fraku = \Lie(U)$.

Let $X$ be the weight lattice of $\Phi$. Write $X^+$ for the set of dominant weights in $X$, and $X_r$ for the set of $p^r$-restricted dominant weights in $X^+$. Given $\lambda\in X^+$, let $L(\lambda)$ be the simple rational $G$-module of highest weight $\lambda$. For each $r\ge 1$, let $F^r: G \to G$ be the $r$-th iterate of the Frobenius morphism of $G$. We call $\ker F_r$ the $r$-th Frobenius kernel of $G$, and denoted by $G_{(r)}$ (to be distinguished with the group of exceptional type $G_2$).

\subsection{Nilpotent orbits} Given a $G$-variety $V$ and a point $v$ of $V$, we denote by $\calO_v$ the $G$-orbit of $v$ (i.e., $\calO_v=G\cdot v$). For example, consider the nilpotent cone $\N$ of $\g$ as a $G$-variety with the adjoint action. There are well-known orbits: $\calO_{\reg}=G\cdot v_{\reg}, \calO_{\subreg}=G\cdot v_{\subreg}$, (we abbreviate it by $\calO_{\sub}$,) and $\calO_{\rm{min}}=G\cdot v_{\rm{min}}$ where $v_{\reg}, v_{\subreg},$ and $v_{\rm{min}}$ are representatives for the regular, subregular, and minimal orbits. Denote by $z(v)$ the centralizer of $v$ in $\g$. For convenience, we write $z_{\reg}$ ($z_{\sub}$ and $z_{\min}$) for the centralizers of $v_{\reg}$ ($v_{\sub}$ or $v_{\min}$).

\subsection{Basic algebraic geometry conventions}\label{algebraic geometry conventions} Let $R$ be a commutative Noetherian ring with identity. We use $R_{\red}$ to denote the reduced ring $R/\sqrt{0}$ where $\sqrt{0}$ is the radical ideal of the trivial ideal $0$, which consists of all nilpotent elements of $R$. Let $\spec R$ be the spectrum of all prime ideals of $R$. If $V$ is a closed subvariety of an affine space $\A^n$, we denote by $I(V)$ the radical ideal of $k[\A^n]=k[x_1,\ldots,x_n]$ associated to this variety. Let $X$ be an affine variety. Then we always write $k[X]$ for the coordinate ring of $X$ which is the same as the ring of global sections $\calO_X(X)$.

\subsection{Commutative algebra}
Consider an $m\times n$ matrix
\[ X=\left( \begin{array}{ccc}
x_{11} & \cdots & x_{1n} \\
\vdots & \ddots & \vdots \\
x_{m1} & \cdots & x_{mn} \end{array} \right) \]
whose entries are independent indeterminates over the field $k$. Let $k[X]$ be the polynomial ring over all the indeterminates of $X$, and let $I_t(X)$ be the ideal in $k(X)$ generated by all $t$ by $t$ minors of $X$. For each $t\ge 1$, the ring 
\[ R_t(X)=\frac{k[X]}{I_t(X)} \]
is called a {\it determinantal ring}. The following is one of the nice properties of determinantal rings.
\begin{proposition}\cite{BV:1988}\label{commutative algebra}
For every $1\le t\le\min(m,n)$, $R_t(X)$ is a reduced, Cohen-Macaulay, normal domain of dimension $(t-1)(m+n-t+1)$.
\end{proposition}
We denote by $D_t(X)$ the determinantal variety defined by $I_t(X)$.

\section{Commuting varieties of centralizers}\label{3 by 3 matrices}
Let $\g=\fraksl_n$. It is easy to see that $C_r(z_{\reg})=z_{\reg}^r$ for every $r\ge 1$. We study in this section the variety $C_r(z_{\sub})$. Then we apply our calculations to compute the dimensions of $C_r(\fraksl_3)$ and $C_r(\mathfrak{gl}_3)$ for each $r\ge 1$.

\subsection{Nice properties of $C_r(z_{\sub})$}

\begin{theorem}\label{commuting variety over z_sub}
For each $r\ge 1$, the variety $C_r(z_{\sub})$ is irreducible, Cohen-Macaulay and normal. Moreover, we have
\[ \dim C_r(z_{\sub})=
\begin{cases}
(n-1)r+2~~~&\text{if}~~~~p\nmid n,\\
nr+1~~~&\text{otherwise}.
\end{cases}
\]
\end{theorem}

\begin{proof}
Without loss of generality, let $v_{\sub}$ be the Jordan matrix corresponding to the partition $[n-1, 1]$. Then an element $u$ of $z_{\sub}$ is of the form
\[ u=\left( \begin{array}{ccccc} a_1 & 0 & 0 & \cdots & 0 \\
a_2 & a_1 & 0 & \ddots & \vdots \\
\vdots & \vdots & \vdots & \ddots & 0 \\
a_{n-1} & \cdots & a_2 & a_1 & c \\
b & 0 & \cdots & 0 & (1-n)a_1 \end{array} \right). \]
By using the multiplication of matrices by blocks, we obtain for any pair $u, u'$ in $z_{\sub}$
\[ [u, u'] = \left( \begin{array}{ccccc} 0 & 0 & 0 & \cdots  & 0 \\
0 & 0 & 0 & \ddots & \vdots \\
\vdots & \vdots & \vdots & \ddots & 0 \\
cb'-bc' & \cdots & \cdots & 0 & n(a_1c'-a_1'c) \\
n(ba_1'-a_1b') & 0 & \cdots & 0 & 0 \end{array} \right). \]
If $p$ does not divide $n$, the defining polynomials of the commutator are $a_1c'-a_1'c, cb'-bc',$ and $ba_1'-b'a_1$. This implies that the variety $C_r(z_{\sub})$ is defined by the collection of polynomials $\{x_iy_j-x_jy_i\ ,\  y_iz_j-y_jz_i\ ,\  x_iz_j-x_jz_i~|~1\le i\le j\le r\}$ in $k[x_i,y_i,z_i, t_{ij}~|~1\le i\le r,~2\le j\le n-1]$. So we can identify $C_r(z_{\sub})$ with the determinantal variety $D_2(X)$ where $X$ is the matrix 
\[ X= \left( \begin{array}{cccc} x_1 & x_2 & \cdots & x_r \\
y_1 & y_2 & \cdots & y_r \\
z_1 & z_2 & \cdots & z_r \end{array} \right). \]
This identification implies the following isomorphism of varieties
\[ C_r(z_{\sub})\cong D_2(X)\times k^{(n-2)r}. \]
On the other hand, if $p$ divides $n$, then the commutator is defined by $cb'-c'b$. This gives us the following
\[ C_r(z_{\sub})\cong D_2(Y)\times k^{(n-1)r}\]
where $Y$ is the matrix of indeterminates defined by
\[  Y=\left( \begin{array}{cccc} x_1 & x_2 & \cdots & x_r \\
y_1 & y_2 & \cdots & y_r \end{array} \right). \]
Hence $D_2(Y)$ is of dimension $r+1$ and so we obtain $\dim C_r(z_{\sub})=(n-1)r+r+1=nr+1$. Other results immediately follow from Proposition \ref{commutative algebra}.
\end{proof}
\begin{remark}\label{remark p=3}
This result can not be generalized to the centralizer of an arbitrary nilpotent element. In particular, there is a nilpotent element $e$ in $\g$ such that $C_2(z(e))$ is reducible (hence so is $C_r(z(e))$) \cite{Ya:2006}.
\end{remark}

\subsection{An application}
Suppose $\g$ is an arbitrary simple Lie algebra. We first study a connection between the dimension of the commuting variety $C_r(\g)$ and that of a certain mixed commuting variety.

Define a mixed commuting variety by 
\[ C(\N,\g^{r-1})=\{(v_1,\ldots,v_r)~|~v_1\in\N,~(v_2,\ldots,v_r)\in\g^{r-1},~[v_i,v_j]=0~\text{with}~1\le i\le j\le r\},\]
which is a subvariety of $C_r(\g)$ with nilpotency condition for the first factor. Note that 
\begin{equation}\label{inequality}
\dim C_r(\g)\le \dim C(\N,\g^{r-1})+\rank(\g). 
\end{equation} 
Hence, knowing the dimension of $C(\N,\g^{r-1})$ allows one to obtain an upper bound on $\dim C_r(\g)$. Note also that in the case $r=2$, Baranovsky used this variety to compute the dimension of $C_2(\N)$, \cite[Theorem 2]{Ba:2001}. Assume in this subsection that $p\ne 3$. We aim to compute the dimensions of $C_r(\fraksl_3)$ and $C_r(\mathfrak{gl}_3)$. We begin with a lemma.

\begin{theorem}\label{dim C_r(sl_3)}
For each $r\ge 2$, we have $\dim C(\N,\fraksl_3^{r-1})=2r+4$ and $\dim C_r(\fraksl_3)=2r+6$.
\end{theorem}

\begin{proof}
Note that $C_r(\fraksl_3)$ contains a component $\overline{G\cdot\mathfrak{t}^r}$, where $\mathfrak{t}$ is a Cartan subalgebra of $\fraksl_3$. It is easy to see that the dimension of this component is $2r+6$. So it suffices to show that $\dim C_r(\fraksl_3)\le 2r+6$. We proceed by induction and assume that $\dim C_{r-1}(\fraksl_3)\le 2(r-1)+6=2r+4$. As $\N$ contains three nilpotent orbits: $\calO_{\reg},\calO_{\sub},$ and $0$, we have the following decomposition
\begin{align}
C(\N,\fraksl_3^{r-1})=\overline{G\cdot (v_{\reg},z_{\reg},\ldots,z_{\reg})}\cup \overline{G\cdot (v_{\sub},C_{r-1}(z_{\sub}))}\cup 0\times C_{r-1}(\fraksl_3).
\end{align}
in which the dimension of each component is given as follows:
\begin{align*}
\dim \overline{G\cdot (v_{\reg},z_{\reg},\ldots,z_{\reg})} &=\dim G\cdot v_{\reg}+\dim z_{\reg}^{r-1}=6+2(r-1)=2r+4, \\
\dim \overline{G\cdot (v_{\sub},C_{r-1}(z_{\sub}))} &=\dim G\cdot v_{\sub}+\dim C_{r-1}(z_{\sub})=4+2(r-1)+2=2r+4
\end{align*}
where $\dim C_{r-1}(z_{\sub})=2(r-1)+2$ by Theorem \ref{commuting variety over z_sub}. So the dimension of $C(\N,\fraksl_3^{r-1})$ is $2r+4$, which confirms the first statement. Then the inequality \eqref{inequality} implies that
\[ \dim C_r(\fraksl_3)\le 2r+4+\rank\fraksl_3=2r+6 \]
which completes our proof. 
\end{proof}

\begin{corollary}
For each $r\ge 2$, we have $\dim C_r(\mathfrak{gl}_3)=3r+6$.
\end{corollary}

\begin{proof}
It follows from Theorem \ref{dim C_r(sl_3)} and \cite[Theorem 4.2.1]{N:2012}. 
\end{proof}

\begin{remark}\label{remark2 p=3}
Note first that our computation does not rely on the irreducibility of $C_r(\mathfrak{gl}_3)$ for each $r\ge 1$. In the case $p=3$, Remark \ref{remark p=3} shows that $C_r(\fraksl_3)$ is of dimension at least $3r+2$. This shows that $C_r(\fraksl_3)$ is reducible when $r>4$.
\end{remark}

\section{Mixed Commuting Varieties over type $A_2$}\label{mixed commuting varieties} We compute in this section the dimension of mixed commuting varieties over various closed sets in $\fraksl_3$. Our calculations are based on the dimension for a certain class of varieties defined by minors of a matrix of indeterminates.   

To begin we set
\[ C_{i,j,m}=C(\underbrace{\overline{\calO_{\sub}},\ldots,\overline{\calO_{\sub}}}_{i~ \text{times}},\underbrace{\N,\ldots,\N}_{j~ \text{times}},\underbrace{\fraksl_3,\ldots,\fraksl_3}_{m~ \text{times}}). \]
The goal is to compute the dimension of $C_{i,j,m}$ for every set of non-negative integers $i,j,m$.
\subsection{} Note that $\dim C_{i,j,m}$ is known in the following cases:
\begin{itemize}
\item If $i=j=0$ then $\dim C_{i,j,m}=\dim C_m(\fraksl_3)=2m+6,$
\item If $i=m=0$ then $\dim C_{i,j,m}=\dim C_j(\N)=2j+4,$
\item If $j=m=0$ then $\dim C_{i,j,m}=\dim C_i(\overline{\calO}_{\sub})=2i+2$
\end{itemize}
by Theorem \ref{dim C_r(sl_3)} and \cite[Theorems 7.1.2 and 7.2.3]{N:2012}. When $i=0$, the dimension of $C_{i,j,m}$ can be easily computed as follows.

\begin{proposition}\label{C_{0,j,m}}
For $j,m\ge 1$, we have $\dim C_{0,j,m}=2(j+m)+4$. Consequently, the variety $C_{0,j,m}$ is never irreducible.  
\end{proposition}

\begin{proof}
Observe that
\[ \dim C_{j+m}(\N)\le \dim C_{0,j,m}\le \dim C(\N,\fraksl_3^{j+m-1}).\]
From earlier we have
\[ 2(j+m)+4\le \dim C_{0,j,m}\le 2(j+m)+4.\]
This gives us the dimension of $C_{0,j,m}$. It indicates that $C_{j+m}(\N)$ is a proper irreducible component of $C_{0,j,m}$. Hence, the reducibility of $C_{0,j,m}$ is proved.
\end{proof}

\subsection{} Fix $v_{\sub}$, the canonical Jordan block matrix corresponding to the partition $[2,1]$ of $3$. Then the centralizer of $v_{\sub}$ in $\fraksl_3$ is
\[ \left\{ \left( \begin{array}{ccc}  x & 0 & 0 \\
y & x & t \\
z & 0 & -2x \end{array} \right) |~x, y,z,t\in k \right\}.
\]
We recall results in \cite{N:2012} on the intersections of $z_{\sub}$ with $\overline{\calO_{\sub}}$ or $\N$ respectively which play important roles in our calculations.
\begin{proposition}\cite[Lemma 7.2.2]{N:2012}\label{review} 
There are identities
\begin{align*}
z_{\sub}\cap\N &=\left\{ \left( \begin{array}{ccc}  0 & 0 & 0 \\
y & 0 & t \\
z & 0 & 0 \end{array} \right) |~ y,z,t\in k \right\}, \\
z_{\sub}\cap\overline{\calO_{\sub}} &=\left\{ \left( \begin{array}{ccc}  0 & 0 & 0 \\
y & 0 & 0 \\
z & 0 & 0 \end{array} \right)~|~y,z\in k\right\} \cup \left\{ 
\left( \begin{array}{ccc}  0 & 0 & 0 \\
y & 0 & t \\
0 & 0 & 0 \end{array} \right)
|~ y,t\in k \right\} \\
&=: V_1 \cup V_2.
\end{align*}
Moreover, if $u,v\in V_1\cup V_2$ then
\[ [u,v]=0 \Leftrightarrow u,v\in V_1\quad\text{or}\quad u,v\in V_2. \]
\end{proposition}

\subsection{Some results on determinantal varieties} Before investigating the dimension of the mixed commuting variety $C_{i,j,m}$, we need to prove some results related to dimensions of determinantal varieties.

\begin{theorem}\label{lemma 2}
Let $X_{i,j,m}$ be the matrix of indeterminates 
\[ \left( \begin{array}{ccccccccc} x_1 & \cdots & x_i & x_{i+1} & \cdots & x_{i+j} & x_{i+j+1} & \cdots & x_{i+j+m} \\
0 & \cdots & 0 & y_1 & \cdots & y_j & y_{j+1} & \cdots & y_{j+m} \\ 
0 & \cdots & 0 & 0 & \cdots & 0 & z_1 & \cdots & z_m \end{array} \right). \]
Then $\dim V(I_2(X_{i,j,m}))=\max\{m+2, j+m+1, i+j+m\}.$
\end{theorem}

\begin{proof}
Let 
\begin{align*}
A &=\{x_1,\ldots,x_i\} \quad,\quad A' =\{y_1,\ldots,y_{j+m},z_1,\ldots,z_m\} \\
B &=\{y_1,\ldots,y_j\} \quad,\quad B' =\{x_1,\ldots,x_i,z_1\ldots,z_m\},\\
C &=\{z_1,\ldots,z_m\} \quad,\quad C' =\{x_1,\ldots,x_{i+j},y_1,\ldots,y_j\}.
\end{align*}
Then we define $XY=\{xy~|~x\in X, y\in Y\}$. It is observed that the sets $AA', BB',$ and $CC'$ are in $I_2(X_{i,j,m})$. These sets of monomials give us the following decomposition
\[ V(I_2(X_{i,j,m}))=V(I_2(X_{i+j+m,0,0}))\cup V(I_2(X_{0,j+m,0}))\cup V(I_2(X_{0,0,m})).\]
By Proposition \ref{commutative algebra}, we have
\begin{align*}
&\dim V(I_2(X_{i+j+m,0,0}))=i+j+m, \\
&\dim V(I_2(X_{0,j+m,0})) =j+m+1,\\
&\dim V(I_2(X_{0,0,m}))=m+2.
\end{align*}
Hence, the result follows.
\end{proof}

This computation can be generalized to calculate the dimension of the determinantal variety $W$ defined by 2 by 2 minors of the matrix
\[ \left( \begin{array}{ccccccccccc} x_{11} & \cdots & x_{1,a_1} & \cdots & \cdots & x_{1,a_2} & \cdots & \cdots & \cdots & \cdots & x_{1,a_m}  \\
\vdots & \cdots & \vdots & \cdots & & & & & & & \vdots \\
x_{b_1,1} & \cdots & x_{b_1,a_1} & x_{b_1,a_1+1} & \cdots & \cdots & \cdots & \cdots & \cdots & \cdots & x_{b_1,a_m} \\
0 & \cdots & 0 & x_{b_1+1,a_1+1} & \cdots & \cdots & \cdots & \cdots & \cdots & \cdots & x_{b_1,a_m-a_1}\\
\vdots & \cdots & \vdots & \vdots & & & & & & & \vdots \\
0 & \cdots & 0 & 0 & \cdots & 0 & x_{b_2+1,a_2+1} & \cdots & \cdots & \cdots & x_{b_2+1,a_m-a_1-a_2} \\
\vdots & \cdots & \vdots & \vdots & \cdots & \vdots & \vdots & \cdots & \cdots & \cdots & \vdots \\
0 & \cdots & 0 & 0 & \cdots & 0 & 0 & \cdots & x_{b_n,a_{m-1}+1} & \cdots & x_{b_n,a_m-\Sigma_{i=1}^{m-1}a_i}
 \end{array} \right). \]
Denote this matrix by $X(a_1,\ldots,a_m,b_1,\ldots,b_n)$. Then we can decompose $W$ into the union of determinantal varieties defined by 2 by 2 minors of the following matrices:
\[ X(a_1,\ldots,a_m,b_1, 0\ldots,0),~ X(0,a_2,\ldots,a_m,b_1,b_2,0\ldots,0),~\ldots,~ X(0,\ldots,0,a_m,b_1,\ldots,b_n). \]  
Hence, we obtain the following result.
\begin{theorem}\label{general matrix} Given $a_1<a_2<\cdots<a_m$, $b_1<b_2<\cdots<b_n$, and the matrix $W$ defined as above. Then
\[\dim V(I_2(W))=\max\{ a_m+b_1+1, a_m-a_1+b_2+1,~\ldots,~a_m-\Sigma_{i=1}^{m-1}a_i+b_n+1\}. \]
\end{theorem}

It would be nice if one can generalize this result for larger minors of $W$.

\subsection{Main Theorem} We can now compute the dimension of the mixed commuting variety $C_{i,j,m}$. If $i=0$, then the answer is obtained from Proposition \ref{C_{0,j,m}}. So we assume that $i\ge 1$.

\begin{theorem}\label{main theorem}
For each $i\ge 1$, let $N_{i,j,m}=\max\{ m+2,j+m+1,i+j+m\}$. Then we have 
\[ \dim C_{i,j,m}=N_{i-1,j,m}+i+j+m+3.
\]
\end{theorem}

\begin{proof}
Let first consider the decomposition
\begin{equation}\label{decompose}
 C_{i,j,m}=G\cdot(v_{\sub},D_{i-1,j,m})\cup 0\times C_{i-1,j,m}
\end{equation}
where $D_{i-1,j,m}=C(\underbrace{z_{\sub}\cap\overline{\calO_{\sub}},\ldots,z_{\sub}\cap\overline{\calO_{\sub}}}_{i-1~ \text{times}},\underbrace{z_{\sub}\cap\N,\ldots,z_{\sub}\cap\N}_{j~ \text{times}},\underbrace{z_{\sub},\ldots,z_{\sub}}_{m~ \text{times}})$. By Proposition \ref{review}, we can further decompose
\begin{align*}
 D_{i-1,j,m} &=C(\underbrace{V_1\cup V_2,\ldots,V_1\cup V_2}_{i-1~ \text{times}},\underbrace{z_{\sub}\cap\N,\ldots,z_{\sub}\cap\N}_{j~ \text{times}},\underbrace{z_{\sub},\ldots,z_{\sub}}_{m~ \text{times}}) \\
&=C(\underbrace{V_1,\ldots,V_1}_{i-1~ \text{times}},\underbrace{z_{\sub}\cap\N,\ldots,z_{\sub}\cap\N}_{j~ \text{times}},\underbrace{z_{\sub},\ldots,z_{\sub}}_{m~ \text{times}})~ \cup \\
& C(\underbrace{ V_2,\ldots, V_2}_{i-1~ \text{times}},\underbrace{z_{\sub}\cap\N,\ldots,z_{\sub}\cap\N}_{j~ \text{times}},\underbrace{z_{\sub},\ldots,z_{\sub}}_{m~ \text{times}}).
\end{align*}
Analyzing the commutators of $V_1$ or $V_2$ with $z_{\sub}\cap\N$ and $z_{\sub}$, we have the following identities
\begin{align*}
 C(\underbrace{V_1,\ldots,V_1}_{i-1~ \text{times}},\underbrace{z_{\sub}\cap\N,\ldots,z_{\sub}\cap\N}_{j~ \text{times}},\underbrace{z_{\sub},\ldots,z_{\sub}}_{m~ \text{times}}) &=V(I)\times k^{i+j+m-1} \\
C(\underbrace{ V_2,\ldots, V_2}_{i-1~ \text{times}},\underbrace{z_{\sub}\cap\N,\ldots,z_{\sub}\cap\N}_{j~ \text{times}},\underbrace{z_{\sub},\ldots,z_{\sub}}_{m~ \text{times}}) &=V(J)\times k^{i+j+m-1} 
\end{align*}
where the affine space $k^{i+j+m-1}$ is from the freeness of $y_1,\ldots,y_{i+j+m-1}$, $I$ and $J$ are respectively the ideals generated by $2\times 2$ minors of the following matrices 
\[ \left( \begin{array}{ccccccccc} z_1 & \cdots & z_{i-1} & z_{i} & \cdots & z_{i+j-1} & z_{i+j} & \cdots & z_{i+j+m-1} \\
0 & \cdots & 0 & t_1 & \cdots & t_j & t_{j+1} & \cdots & t_{j+m} \\ 
0 & \cdots & 0 & 0 & \cdots & 0 & x_1 & \cdots & x_m \end{array} \right) \]
\[ \left( \begin{array}{ccccccccc} t_1 & \cdots & t_{i-1} & t_{i} & \cdots & t_{i+j-1} & t_{i+j} & \cdots & t_{i+j+m-1} \\
0 & \cdots & 0 & z_1 & \cdots & z_j & z_{j+1} & \cdots & z_{j+m} \\ 
0 & \cdots & 0 & 0 & \cdots & 0 & x_1 & \cdots & x_m \end{array} \right). \]
Theorem \ref{lemma 2} gives us 
\[ \dim V(I)=\dim V(J) = N_{i-1,j,m} \]
so that
\[ \dim D_{i-1,j,m}=i+j+m-1+N_{i-1,j,m}.
\]
Hence we have 
\[ \dim G\cdot(v_{\sub},D_{i-1,j,m})=4+\dim D_{i-1,j,m}=i+j+m+3+N_{i-1,j,m}.
\]
It remains to prove that $\dim 0\times C_{i-1,j,m}\le i+j+m+3+N_{i-1,j,m}$. Indeed, if $i=1$, then it equals to $2(j+m)+4$ (by Proposition \ref{C_{0,j,m}}) which is $\le i+j+m+3+N_{i-1,j,m}$ since $N_{i-1,j,m}\ge j+m+1$. For $i>1$, we have by induction that 
\[ \dim C_{i-1,j,m}=i-1+j+m+3+N_{i-2,j,m} \le i+j+m+3+N_{i-1,j,m}. \]
Finally, we have shown that $\dim C_{i,j,m}=i+j+m+3+N_{i-1,j,m}$ as desired.
\end{proof}

Next we show that mixed commuting varieties are usually not irreducible. We start with a lemma.

\begin{lemma}
For each $i\ge 2$, the variety $C_{i,0,0}=C_i(\overline{\calO_{\sub}})$ is reducible of dimension $2r+2$.
\end{lemma}     
\begin{proof}
This is just a corollary of \cite[Theorem 7.2.3]{N:2012}. In particular, we have the following irreducible decomposition
\[ C_{i,0,0}=\overline{G\cdot(v_{\sub},V_1,\ldots,V_1)}\cup \overline{G\cdot(v_{\sub},V_2,\ldots,V_2)} \]
where $V_1$ and $V_2$ are defined in Proposition \ref{review}.
\end{proof}

\begin{theorem}\label{irreduciblity}
For each $i,j,m\ge 0$, the mixed commuting variety $C_{i,j,m}$ is irreducible if and only if $i,j,m$ satisfy one of the following conditions:
\begin{enumerate}
\item $i=j=0$,
\item $i=m=0$,
\item $i=1,~j=m=0$.
\end{enumerate}
\end{theorem}

\begin{proof}
The conditions (1), (2), and (3) in the theorem are equivalent to the cases in which the variety $C_{i,j,m}$ is either $C_m(\fraksl_3),$ or $C_j(\N)$, or $\overline{\calO_{\sub}}$. It is known that these varieties are irreducible. Indeed, the variety $C_j(\N)$ is irreducible by Theorem 7.1.2 in \cite{N:2012}, and the variety $C_m(\fraksl_3)$ is irreducible by Theorem 4.2.1 in \cite{N:2012} and the fact that $C_m(\mathfrak{gl}_3)$ is irreducible \cite{KN:1987}\cite{Gu:1992}. From the decomposition \eqref{decompose} in Theorem \ref{main theorem}, we have
\[  C_{i,j,m}=\overline{G\cdot(v_{\sub},D_{i-1,j,m})}\cup 0\times C_{i-1,j,m} \]
where $\overline{G\cdot(v_{\sub},D_{i-1,j,m})}=G\cdot(v_{\sub},D_{i-1,j,m})\cup\{0\}$. So if $i\ge 1$ and $j\ne 0$ or $m\ne 0$, we always have $C_{i,j,m}$ is reducible. The case $i>1$ and $j=m=0$ was proved in the previous lemma. On the other hand, Proposition \ref{C_{0,j,m}} shows that $C_{i,j,m}$ is reducible in the case $i=0$ and $j,m\ge 1$.
\end{proof}

\begin{remark}
This result also indicates that mixed commuting varieties are rarely normal as all irreducible components contain the origin. In other words, if a mixed commuting variety is reducible, it is not normal. This behavor is analogous with that of mixed commuting varieties over $\fraksl_2$ and its nullcone in \cite[Proposition 6.1.1]{N:2012}. 
\end{remark}

\section{Mixed commuting varieties for other types of rank 2}\label{other rank 2 groups}

We compute in this section the dimension of mixed commuting varieties 
\[ C_{i,j}=C(\underbrace{\overline{\calO_{\sub}},\ldots,\overline{\calO_{\sub}}}_{i~ \text{times}},\underbrace{\N,\ldots,\N}_{j~ \text{times}}) \]
and then determine when they are irreducible for $\g$ of either type $C_2$ or $G_2$. We proceed the computation case by case.
 
\subsection{Type $C_2$ (or $B_2$)} Fix $\g=\mathfrak{sp}_4$ and assume $p>3$. Note that nilpotent orbits of $\g$ are parametrized by partitions $[4], [2,2], [2,1,1],$ and $[1^4]$. So $\overline{\calO_{\sub}}=\overline{\calO_{[2,2]}}$, which includes all the square zero elements in $\N$. This observation will be handy in the following.

\begin{theorem}\label{Type C_2}
For each $r\ge 2$ and $0\le n\le r$, we have
\[ \dim C_{n,r-n}=
\begin{cases}
2r+6 ~~&\text{if}~~$r=2, n=0$,\\
3r+3 ~~&\text{otherwise}.
\end{cases}
\]
In particular, $C_{n,r-n}$ is irreducible if and only if $r=2$ and $n=0$.
\end{theorem}

\begin{proof}
It is easy to see that for $0\le n\le r$ one has
\[ C_r(\overline{\calO_{\sub}})\subseteq C_{n, r-n}\subseteq C_r(\N) \]
It is well known (from the result of Premet) that $\dim C_{0, 2}=\dim C_2(\N)=\dim\g=10$. So it suffices to prove for remaining cases that $\dim C_{n, r-n}=3r+3$ for all $r\ge 2$. Note from earlier observation and Theorem 3.3.3 in \cite{N:2014} that
\[ \dim C_r(\overline{\calO_{\sub}})=3(r+1). \]
On the other hand, \cite[Corollary 3.4.2]{N:2014} implies that $\dim C_r(\N)=3r+3$. This completes our argument on the dimension of $C_{n, r-n}$.

It is easy to see that $C_{0,2}$ is irreducible, again by the result of Premet. Now suppose that $C_{n, r-n}$ is irreducible, and suppose that $r>2$ or $n>0$. Then from the above we have 
\[ \dim C_{n, r-n}=3r+3.\]
This indicates that the variety $G\cdot\mathfrak{w}^r$ is an irreducible component of $C_{n, r-n}$ (see \cite[Proposition 3.2.4]{N:2014}), where $\mathfrak{w}$ is the space of matrices of the form (1) in \cite[3.2]{N:2014}. Apparently, this component is not equal to $C_r(\overline{\calO_{[2,2]}})\subseteq C_{n, r-n}$. Therefore, the variety $C_{n, r-n}$ is reducible. 
\end{proof}

\subsection{Type $G_2$} We assume in this case $p>2(h-1)$. First we calculate the dimension of the nilpotent commuting variety as follows.

\begin{proposition} 
For each $r\ge 2$, one has
\[ \dim C_r(\N)=3r+8. \]
\end{proposition}
\begin{proof}
Fix $r\ge 2$, we decompose
\[ C_r(\N)=\bigcup_{v\in\N/G}\overline{G\cdot(v,C_{r-1}(z(v)\cap\N))} \]
where $\N/G$ is the set of representatives for nilpotent orbits of $\g$. In terms of Bala-Carter labels, $\N$ consists of nilpotent orbits namely $G_2, G_2(a_1), \tilde{A_1},$ and $A_1$. As the first two orbits are distinguished, we have $C_{r-1}(z(v)\cap\N)=C_{r-1}(z(v))$ for $v\in G_2$ or $v\in G_2(a_1)$. Suppose $v$ is in $G_2$, i.e. $v$ is regular nilpotent element. Then $C_{r-1}(z(v))=z(v)^{r-1}$, so that
\[ \overline{G\cdot(v,C_{r-1}(z(v)\cap\N))}=\dim\calO_v+(r-1)\dim(z(v))=2r+10. \]
Now if $v$ is in $G_2(a_1)$, then using the computation in \cite{LT} we can choose
\[ v=e_{\alpha_2}+e_{3\alpha_1+\alpha_2}, \]
here $e_{\alpha}$ denotes the root vector of $\alpha$. Hence
\[ z(v)=\left<e_{\alpha_1+\alpha_2}, e_{2\alpha_1+\alpha_2}, v, e_{3\alpha_1+2\alpha_2}\right>. \]
Consider $r-1$ elements in $z(v)$ as follows
\[ u_i=a_ie_{\alpha_1+\alpha_2}+b_ie_{2\alpha_1+\alpha_2}+c_iv+d_ie_{3\alpha_1+2\alpha_2} \]
for $1\le i\le r-1$. Then for each $1\le i\ne j\le r-1$, we have 
\[ [u_i, u_j]=0\quad\Rightarrow\quad a_ib_j-a_jb_i=0\]
The last equations indicate that the commuting variety $C_{r-1}(z(v))$ is the product of a $(2r-2)$-dimensional space and a determinantal variety of $2\times 2$ minors of the matrix
\[ \left(
\begin{array}{ccc}
a_1 & \ldots & a_{r-1}\\
b_1 & \ldots & b_{r-1}
\end{array} \right).
\]
Hence, $\dim C_{r-1}(z(v))=3r-2$. It follows that
\[  \dim\overline{G\cdot(v,C_{r-1}(z(v)\cap\N))}=3r+8. \] 

Next, we suppose $v$ is in $\tilde{A}_1$. Again, using \cite[p. 65]{LT}\footnote{Note that the tables in pages 64-65 are incorrect.} we can choose $v=e_{\alpha_1}$ and then
\[ z(v)=\left<v, e_{3\alpha_1+\alpha_2}, e_{3\alpha_1+2\alpha_2}, f_{\alpha_2}, f_{3\alpha_1+2\alpha_2}, h_{\alpha_1}+2h_{\alpha_2}\right>.\] 
Note that $f_{\alpha}$ denotes the root vector of $-\alpha$. For each $1\le i\le r-1$, set
\[ u_i=a_iv+b_ie_{3\alpha_1+\alpha_2}+ c_ie_{3\alpha_1+2\alpha_2}+d_if_{\alpha_2}+ m_if_{3\alpha_1+2\alpha_2}+n_i(h_{\alpha_1}+2h_{\alpha_2}).\]
Then simple calculations give us
\begin{align}\label{equations of z(Atilde)}
 [u_i, u_j]=0 \quad\Leftrightarrow\quad 
\begin{cases}
b_im_j-b_jm_i+n_id_j-n_jd_i=0\\
b_in_j-b_jn_i+c_jd_i-c_id_j=0\\
c_jm_i-c_im_j=0\\
2(c_in_j-c_jn_i)=0\\
2(m_jn_i-m_in_j)=0.
\end{cases}
\end{align}
On the other hand, the nilpotency of each $u_i$ requires that $n_i^2+d_im_i=0$ for each $1\le i\le r-1$. To estimate the dimension of $C_{r-1}(z(v)\cap\N)$, we consider a tuple $(u_1,\ldots,u_{r-1})$. If $n_i=0$ for all $i$, then the last equations imply that $m_i=0$ for all $i$. Then the equations in RHS of \ref{equations of z(Atilde)} become $c_jd_i-c_id_j=0$ for all $i\ne j$. It follows that the tuple $(u_1,\ldots,u_{r-1})$ belongs to the product $P_1$ of an affine space of dimension $2(r-1)$ (due to the freeness of $a_i, b_i$) and a determinantal variety of $2\times 2$ minors over the matrix
\[ \left(
\begin{array}{ccc}
c_1 & \ldots & c_{r-1}\\
d_1 & \ldots & d_{r-1}
\end{array} \right).
\]
Otherwise, suppose there is some $n_i\ne 0$, say $n_1\ne 0$, so that $m_1\ne 0$. Since $\chr(k)>2$, the last three equations in \ref{equations of z(Atilde)} determine a determinantal variety of $2\times 2$ minors over the matrix
\[ \left(
\begin{array}{ccc}
c_1 & \ldots & c_{r-1}\\
m_1 & \ldots & m_{r-1}\\
n_1 & \ldots & n_{r-1}
\end{array} \right).
\]
Moreover, the $d_i$'s can then be determined as follows
\begin{align*}
d_1 &=\frac{-n_1^2}{m_1}, \\
d_i &=\frac{1}{n_1}(b_im_1-b_1m_i+n_id_1)
\end{align*}
for all $i\ne 1$. This shows that the tuple $(u_1,\ldots,u_{r-1})$ belongs to the product $P_2$ of an affine space of dimension $2(r-1)$ (due to the freeness of $\{a_i, b_i\}$) and the above determinantal variety. Finally, one has $C_{r-1}(z(v)\cap\N)\subseteq P_1\cup P_2$. It is easy to see that $\dim P_1=3r-2$ and $\dim P_2=3r-1$. Therefore,
\[ \dim\overline{G\cdot(v,C_{r-1}(z(v)\cap\N))}\le 3r+7. \]

Finally, if $v=e_{\alpha_2}$, i.e. $v$ is in the minimal nilpotent orbit of $\g_2$, then the centralizer of $v$ is 
\[ z(v)=\left<v, e_{\alpha_1+\alpha_2}, e_{2\alpha_1+\alpha_2}, e_{3\alpha_1+2\alpha_2}, f_{\alpha_1}, f_{2\alpha_1+\alpha_2}, f_{3\alpha_1+\alpha_2}, 2h_{\alpha_1}+3h_{\alpha_2}\right>.\]
Computing the dimension of $C_{r-1}(z(v)\cap\N)$ is rather complicated in this case. However, we can reduce our calculations to $C_{r-1}(z(v)\cap\overline{\calO_{v}})$. Indeed, generalizing the argument in \cite[Proposition 2.1]{Pr:2003}, we have $GL_{r}(k)$ acting on $C_r(\N)$ as follows
\[ \left(
\begin{array}{cccc}
a_{11} & a_{12} & \cdots & a_{1r} \\
a_{21} & a_{22} & \cdots & a_{2r} \\
: & : & : & : \\
a_{r1} & a_{r2} & \cdots & a_{rr}
\end{array}\right)\bullet(v_1,\ldots,v_r)=\left( \sum_{i=1}^ra_{1i}v_i, \ldots, \sum_{i=1}^ra_{ri}v_i\right) \]
for all $(v_1,\ldots,v_r)\in C_r(\N)$. In particular, suppose $V$ is an irreducible component of $C_r(\N)$. Then any permutation of $(v_1,\ldots,v_r)\in V$ is also in $V$. This indicates that if $V\subseteq\overline{G\cdot(v,C_{r-1}(z(v)\cap\N))}$ then $V$ must be in $C_r(\overline{\calO_v})$. Next, suppose 
\[ u= av+ be_{\alpha_1+\alpha_2}+ce_{2\alpha_1+\alpha_2}+de_{3\alpha_1+2\alpha_2}+ mf_{\alpha_1}+nf_{2\alpha_1+\alpha_2}+sf_{3\alpha_1+\alpha_2}+t(2h_{\alpha_1}+3h_{\alpha_2})\in z(v).\]
Then using GAP we can write down a matrix representation $M_u$ of $u$, (this is where we need the characteristic $p$ of $k$ is sufficiently large). In order for $u$ to be in the minimal orbit, one needs $M_u^2=0$. So we have $c=n=t=0$, $b^2+dm=0$, and $ds-bm=0$. Thus, 
\[ \dim(z(v)\cap\overline{\calO_v})\le 3 \]
and hence
\[ \dim C_r(\overline{\calO_v})=\dim\overline{G\cdot(v,C_{r-1}(z(v)\cap\overline{\calO_v}))}\le 3r+3. \]
In conclusion, we have shown that $\dim C_r(\N)=3r+8$ for each $r\ge 2$.
\end{proof}
\begin{theorem}\label{Type G_2}
For each $r\ge 2$ and $0\le n\le r$, we have
\[ \dim C_{n,r-n}=3r+8. \]
\end{theorem}
\begin{proof}
As we have seen from the above proposition, the variety $\overline{G\cdot(v_{\sub},C_{r-1}(z(v_{\sub})\cap\N))}$ is an irreducible component of maximal dimension in $C_r(\N)$. Therefore,
\[ \dim C_r(\overline{\calO_{\sub}})=\dim C_{n, r-n}=\dim C_r(\N)=3r+8, \]
proving our theorem.
\end{proof}

\section{General case}\label{general result}
We assume now that $\g$ is a simple Lie algebra. Our calculations for rank 2 Lie algebras in the previous section indicate the following
\[ \dim C_{i,j}=\dim C_{i+j}(\N) \]
for all $i, j\ge 0$. We prove here that the statement is true for large enough $i+j$. In particular, we have

\begin{proposition}\label{generalization}
Suppose $\g$ is classical simple Lie algebra of rank $\ell\ge 3$. Suppose also that the characteristic $p$ of $k$ is a good prime for $\g$. Then for the values of $r$ such that
\[
r>
\begin{cases}
3~~&\text{if}~~~\g=\fraksl_{2\ell+1},\\
3+\frac{2}{\ell-1}~~&\text{if}~~~\g=\fraksl_{2\ell},\\
3+\frac{4}{\ell-1}~~&\text{if}~~~\g=\mathfrak{sp}_{2\ell},\\
3+\frac{4}{\ell-3}~~&\text{if}~~~\g=\mathfrak{so}_{2\ell},\\
3+\frac{8}{\ell-3}~~&\text{if}~~~\g=\mathfrak{so}_{2\ell+1},\\
\end{cases}
\]
we have $\dim C_{n,r}=\dim C_{r}(\N)$ for each $n\ge 0$.
\end{proposition}
\begin{proof}
It suffices to prove that for such values of $r$, $\dim C_r(\N)=\dim C_r(\overline{\calO_{\sub}})$. Indeed, suppose $V$ is an irreducible component of maximal dimension in $C_r(\N)$. Let $(v_1,\ldots, v_r)$ be an element in $V$. If $v_i$ is regular for some $i$, then $V$ must be $\overline{G\cdot(v_{\reg},z_{\reg},\ldots,z_{\reg})}$. However, this component is not of maximal dimension in $C_r(\N)$, see Corollary 3.2.5 in \cite{N:2014}. Therefore, $V\subseteq C_r(\overline{\calO_{\sub}})$, which proves our proposition. 
\end{proof}

\begin{remark}
We claim that the proposition also holds for exceptional Lie algebras. Note from the previous section that it holds for type $G_2$. The strategy for verification should be the same. Due to the limited length of this paper, we leave it for interested readers.  
\end{remark}

\section{Applications to support varieties for Frobenius kernels}\label{support varieties}

\subsection{Support varieties} Let $G$ be a simply-connected simple algebraic group defined over $k$ (we assume that $p\ge 3$ in this section). For each $r\ge 1$, recall that $G_{(r)}$ denotes the $r$-th Frobenius kernel of $G$. Then set
\[ \opH^{\bullet}(G_{(r)},k)=\bigoplus_{i\ge 0}\opH^i(G_{(r)},k)\quad,\quad \opH^{2\bullet}(G_{(r)},k)=\bigoplus_{i\ge 0}\opH^{2i}(G_{(r)},k). \] 
Under the cup product, $\opH^{2\bullet}(G_{(r)},k)$ is a commutative ring. Given a finite dimensional $G$-module $M$, we consider Ext$^\bullet_{G_{(r)}}(M,M)$ as a $\opH^{2\bullet}(G_{(r)},k)$-module with the action induced by the cup product. Then the support variety of $M$, denoted by $V_{G_{(r)}}(M)$, is the variety of the annihilator of Ext$^\bullet_{G_{(r)}}(M,M)$ in the ring $\opH^{2\bullet}(G_{(r)},k)$. The dimension of this variety can be identified with the {\it complexity} $c_{G_{(r)}}(M)$ of $M$ over $G_{(r)}$, which is defined to be the growth rate of the sequence $\{\text{Ext}^i_{G_{(r)}}(M,M)\}_{i\ge 0}$ (see \cite{NPV:2002} for further details). Note that 
\[ V_{G_{(r)}}(k)=\spec\opH^{2\bullet}(G_{(r)},k)_{\red}=C_r(\N)\]
when $p\ge h$ \cite[Theorem 5.2]{BFS1:1997}, \cite[1.1]{CLN:2008}. 

\subsection{Connection to mixed commuting varieties} For sufficiently large values of $p$, one can compute the support variety of the simple $G$-module $L(\lambda)$ as follows.

\begin{proposition}\cite[Theorem 3.2]{So:2012}
Let $G$ be a classical simple algebraic group. Suppose that $p>hc$. Let $\lambda$ be a weight with $\lambda=\lambda_0+p\lambda_1+\cdots+p^q\lambda_q$, $\lambda_i\in X_1(T)$. Then for each $r\ge 1$, we have
\[ V_{G_{(r)}}(L(\lambda))=\{(\beta_0,\ldots,\beta_{r-1})\in C_r(\N)~|~\beta_i\in V_{G_{(1)}}(L(\lambda_{r-i-1})) \}. \]
\end{proposition}
Here the number $c$ is an integer defined for every type of $G$ as in \cite[Section 3]{So:2012}. Recall that if $G$ is of type $A_\ell$ ($B_\ell, C_\ell,$ or $D_\ell$), then $c=\left(\frac{n+1}{2}\right)^2$ ($\frac{\ell(\ell+1)}{2}, \frac{\ell^2}{2}$, or $\frac{\ell(\ell-1)}{2}$). Let $\lambda$ be a weight in $X$ and set 
\[ \Phi_{\lambda,p}=\{\beta\in\Phi~|~(\lambda+\rho,\beta^\vee)\in p\mathbb{Z}\}.\]
Suppose that the prime $p$ is good for $G$. We recall that $\lambda$ is called $p$-regular if $\Phi_{\lambda,p}=\emptyset$, otherwise it's called $p$-singular. In the case that the Lusztig character formula holds for all restricted dominant weights, a result of Drupieski, Nakano, and Parshall \cite[Theorem 4.1]{DNP:2012} can be used to give an explicit description for $V_{G_{(r)}}(L(\lambda))$.

\begin{proposition}\label{connection}
Let $G$ be a classical simple algebraic group. Suppose that $p>hc$ and assume that the Lusztig character formula\footnote{See formula (4.0.1) in \cite{DNP:2012}} holds for all restricted dominant weights. Then for $\lambda\in X^+$ with $\lambda=\lambda_0+p\lambda_1+\cdots+p^q\lambda_q$, $\lambda_i\in X_1(T)$,
\[ V_{G_{(r)}}(L(\lambda))=\{(\beta_0,\ldots,\beta_{r-1})\in C_r(\N)~|~\beta_i\in G\cdot\fraku_{J_{r-i-1}} \} \]
where $J_i\subset\Pi$ such that $w(\Phi_{\lambda_i, p})=\Phi_J$ for each $i$\footnote{Let $\lambda=w\cdot\lambda^-, \lambda^-\in\overline{C}_{\mathbb{Z}}^-, w\in W_p$ and $w$ is minimal dominant for $\lambda^-$.}.
\end{proposition}

This result shows that $V_{G_{(r)}}(L(\lambda))$ is a mixed commuting variety where every component of an $r$-tuple is in a closed subvariety of the nilpotent cone $\N$. So it is potential to be reducible for most of the cases. In particular, the complexity of $L(\lambda)$ can be determined in the following cases.

\begin{proposition}
Keep the assumptions as in the above proposition. Suppose further that $\rank(G)\ge 3$ and for each $0\le i\le r-1$,  $J_i$ is either $\Pi$ or $\Pi\backslash\{\alpha_i\}$ for some simple root $\alpha_i$. Then one has $c_{G_{(r)}}(L(\lambda))=c_{G_{(r)}}(k)$ for all $r>11$.
\end{proposition}

\begin{proof}
Observe that we then have for each $i$
\[ G\cdot\fraku_{J_i}=
\begin{cases}
\N~~~&\text{if}~~~J_i=\Pi,\\
\overline{\calO_{\sub}}~~~&\text{otherwise}.
\end{cases}
\]
It follows that
\[ V_{G_{(r)}}(L(\lambda))=C(\underbrace{\overline{\calO_{\sub}},\ldots,\overline{\calO_{\sub}}}_{a~\text{times}}, \underbrace{\N,\ldots,\N}_{b~\text{times}})=C_{a,b}\] 
where $a$ is the number of $i$ such that $J_i=\Pi\backslash\{ \alpha_i\}$ for some simple root $\alpha_i$, and $b=r-a$. Proposition \ref{generalization} implies that for $r\ge 11$
\[ c_{G_{(r)}}(L(\lambda))=\dim V_{G_{(r)}}(L(\lambda))=\dim C_r(\N)=c_{G_{(r)}}(k), \]
which completes our proof.
\end{proof}

\subsection{Rank 2 case} Assume that $G$ is of type $A_2$ or type $C_2$. We can have more explicit information about $V_{G_{(r)}}(L(\lambda))$ in these specific cases. In particular, the following results classifies the behaviors of support varieties for $L(\lambda)$ up to the regularity of $\lambda$.

\begin{proposition}\cite[Corollary 6.6.1]{NPV:2002}
Let $G$ be a simple algebraic group of rank 2 and $p$ good.
\begin{enumalph}
\item If $\lambda$ is a $p$-singular weight, then $V_{G_{(1)}}(L(\lambda))=\overline{\calO_{\sub}}$.
\item If $\lambda$ is a $p$-regular weight, then $V_{G_{(1)}}(L(\lambda))=V_{G_{(1)}}(k)=\N$.
\end{enumalph}
\end{proposition}

Let $\lambda$ be a dominant weight and $\lambda=\lambda_0+p\lambda_1+\cdots+p^q\lambda_q$ with  $\lambda_i\in X_1(T)$. Let $a_{\lambda_r}, b_{\lambda_r}$ be the number of singular weights and regular weights respectively in $\{ \lambda_0,\ldots,\lambda_{r-1}\}$. Then by the propositions above, we have for each $r\ge 1$
\[ V_{G_{(r)}}(L(\lambda))=C(\underbrace{\overline{\calO_{\sub}},\ldots,\overline{\calO_{\sub}}}_{a_{\lambda_r}~\text{times}}, \underbrace{\N,\ldots,\N}_{b_{\lambda_r}~\text{times}})=C_{a_{\lambda_r},b_{\lambda_r},0}. \] 
This is a special case of mixed commuting varieties appearing in previous sections. Hence we know the dimensional and irreducible behaviors of the support variety $V_{G_{(r)}}(L(\lambda))$.

\begin{theorem}\label{support variety 1}
Let $G$ be a classical simple algebraic group of rank 2. Suppose $p>6$ (or $p>8$) if $G$ is of type $A_2$ (or $C_2$). Suppose $\lambda$ is a weight in $X^+$. Then for each $r\ge 2$, the support variety $V_{G_{(r)}}(L(\lambda))$ can be identified with the mixed commuting variety $C_{a_{\lambda_r},b_{\lambda_r},0}$. Furthermore, if $G$ is of type $A_2$ then

\[ c_{G_{(r)}}(L(\lambda))=\dim V_{G_{(r)}}(L(\lambda))=
\begin{cases}
2b_{\lambda_r}+4 ~~&\text{if}~~~~a_{\lambda_r}=0, \\
2(a_{\lambda_r}+b_{\lambda_r})+3 ~~&\text{if}~~~~a_{\lambda_r}=1, \\
2(a_{\lambda_r}+b_{\lambda_r})+2~~&\text{if}~~~~a_{\lambda_r}>1. 
\end{cases}
\]
Otherwise, 
\[ c_{G_{(r)}}(L(\lambda))=\dim V_{G_{(r)}}(L(\lambda))=
\begin{cases}
2b_{\lambda_r}+6 ~~&\text{if}~~a_{\lambda_r}=0, b_{\lambda_r}=2,\\
3(a_{\lambda_r}+b_{\lambda_r})+3 ~~&\text{otherwise}.
\end{cases}
\]
\end{theorem}

\begin{theorem}\label{support variety 2}
Under the same assumption as in the previous theorem. For each $r\ge 2$, 
\begin{itemize}
\item if $G$ is of type $A_2$ then the support variety $V_{G_{(r)}}(L(\lambda))$ is irreducible if and only if $a_{\lambda_r}=0$. In other words, $V_{G_{(r)}}(L(\lambda))$ is irreducible if and only if there is no singular weight in the decomposition of $\lambda$.
\item If $G$ is of type $C_2$ then the support variety $V_{G_{(r)}}(L(\lambda))$ is irreducible if and only if $a_{\lambda_r}=0$ and $b_{\lambda_r}=2$.
\end{itemize}
\end{theorem}

\begin{proof}
It immediately follows from Theorem \ref{support variety 1} and Theorems \ref{irreduciblity} and \ref{Type C_2}.
\end{proof}

\section*{Acknowledgments}

The author deeply acknowledges the discussions with Daniel K.~Nakano and Christopher M.~Drupieski on support varieties. He thanks Christopher Bendel for useful suggestions. He is also grateful for the comments of the referee.

\providecommand{\bysame}{\leavevmode\hbox to3em{\hrulefill}\thinspace}

\end{document}